\documentclass[12pt,a4paper]{amsart}
\usepackage[dvipsnames]{xcolor}
\usepackage{amsfonts}
\usepackage{amscd}
\usepackage[latin1]{inputenc}
\usepackage{t1enc}
\usepackage[mathscr]{eucal}
\usepackage{indentfirst}
\usepackage{graphicx}
\usepackage{graphics}
\usepackage{pict2e}
\usepackage{epic}
\numberwithin{equation}{section}
\usepackage[margin=2.6cm]{geometry}
\usepackage{epstopdf} 
\usepackage{amsmath}
\allowdisplaybreaks
\usepackage{amssymb}
\usepackage{amsthm}
\usepackage{bbm}
\usepackage{tikz}
\usepackage{paralist}
\usepackage[onehalfspacing]{setspace}
\usepackage{amsfonts}
\usepackage{color}
\usepackage{nicefrac}
\usepackage{mathrsfs}
\usepackage{multirow}
\usepackage{mathtools}
\usepackage{tikz-cd}

\newcounter{dummy}
\usepackage{hyperref}
\usepackage{enumitem}
\makeatletter
\newcommand\myitem[1][]{\item[#1]\refstepcounter{dummy}\def\@currentlabel{#1}}
\makeatother
\newtheorem{thm}{Theorem}
\numberwithin{thm}{section}
\newtheorem{lemma}[thm]{Lemma}

\newtheorem{definition}[thm]{Definition}
\newtheorem{coro}[thm]{Corollary}

\newtheorem*{thm*}{Theorem}
\newtheorem*{prop*}{Proposition}
\numberwithin{equation}{section}
\theoremstyle{remark}
\newtheorem{remark}[thm]{Remark}
\newtheorem{ex}[thm]{Example}

\newcommand{\A}{\mathscr{A}}
\newcommand{\B}{\mathscr{B}}
\newcommand{\E}{\mathcal{E}}

\newcommand{\R}{\mathbb{R}}
\newcommand{\N}{\mathbb{N}}
\newcommand{\Z}{\mathbb{Z}}

\newcommand{\C}{\mathbb{C}}

\newcommand{\T}{\mathbb{T}}

\newcommand{\D}{\mathscr{D}}

\newcommand{\test}{\mathcal{T}_{\mathcal{A}}}

\newcommand{\dx}{\,\textup{d}x}

\everymath{\displaystyle}

\DeclareMathOperator{\Lin}{Lin}

\DeclareMathOperator{\curl}{curl}

\DeclareMathOperator{\dist}{dist}
\DeclareMathOperator{\spann}{span}
\DeclareMathOperator{\divergence}{div}

\renewcommand{\phi}{\varphi}

\newcommand{\M}{\mathscr{M}}

\definecolor{Gump}{rgb}{0,0.6,0.4}
\definecolor{Hanks}{rgb}{0.7,0.3,0.1}

\begin{document}
\title{$\mathscr{A}$-free truncation and higher integrability of minimisers}
\author[Schiffer]{Stefan Schiffer}
\address{Max-Planck Institute for Mathematics in the Sciences} 
\email{stefan.schiffer@mis.mpg.de}
\begin{abstract}
    We show higher integrability of minimisers of functionals
    \[
    I(u) = \int_{\Omega} f(x,u(x)) ~\textup{d}x
    \]
    subject to a differential constraint $\mathscr{A} u=0$ under natural $p$-growth and $p$-coercivity conditions for $f$ and regularity assumptions on $\Omega$. For the differential operator $\mathscr{A}$ we asssume a rather abstract truncation property that, for instance, holds for operators $\mathscr{A}=\mathrm{curl}$ and $\A=\mathrm{div}$. The proofs are based on the comparison of the minimiser to the truncated version of the minimiser.
\end{abstract}
\maketitle
\section{Introduction}
Functionals of the form 
\begin{equation}\label{intro:functional:grad}
    J(v) = \int_{\Omega} f(x,v(x),\nabla v(x)) \dx
\end{equation}
have been intensively studied for a long time. \emph{Existence} of minimisers often is obtained through growth condition from above and below (coercivity) and assuming \emph{quasiconvexity} (cf. \cite{Morrey}) of the integrand, guaranteeing weak lower-semiconinuity. Minimisers then are part of the function space $W^{1,p}$ when assuming $p$-growth and $p$-coercivity (cf. \eqref{intro:growth}); the question of regularity theory is whether we can show anything more.

As a very basic example, if we take the integrand $f(x,v,A) = \vert A \vert^2$ minimisers are harmonic functions and therefore have very strong regularity properties. For general $f$ the situation is, of course, more complicated. We would like to mention two possible results: First, higher \emph{differentiability} aims to show that minimisers exhibit better differentiability properties than just $W^{1,p}$ such as $W^{1+\varepsilon,p}$ or $Du \in C^{0,\alpha}_{\mathrm{loc}}(\Omega \setminus \Omega^s)$ for some singular set $\Omega^s$ of low dimension. Higher \emph{integrability} focuses on showing that $u \in W^{1,r}$ for some $r>p$. Whether such higher regularity result can be achieved obviously depends on the assumptions on the integrand $f$.

There has been very detailed study of various situations in which \eqref{intro:functional:grad} appears and it is impossible to give an exhaustive, but short overview over all obtained results. We therefore refer the reader to the works \cite{AS,Evans,GG,Giusti,Mingione, Mingione2} and the references therein for a more detailed overview of the achievements in the gradient case.

\medskip

Naturally, less well-studied is the twin sibling of functional \eqref{intro:functional:grad}, which reads
\begin{equation} \label{intro:functional}
    I(u) := \begin{cases}
        \int_{\T_n} f(x,u(x)) \dx & \text{if } \A u =0, \\
        \infty & \text{else.}
    \end{cases}
\end{equation}
Here, $u \colon \T_n \to \R^m$ is a function on the $n$-torus, $f$ is a measurable  integrand that satisfies some $p$-growth condition from above and below (cf. \eqref{intro:growth} below) and $\A$ is a constant coefficient differential operator 
\[
\A u = \sum_{\vert \alpha \vert =k} A_{\alpha} \partial^{\alpha} u.
\]
To avoid a detailed account of boundary values, in this introduction and for the majority of the paper, we restrict the treatment of the functional to the $n$-torus $\T_n$ instead of dealing with a general open subset $\Omega \subset \R^n$. We refer to Corollary \ref{coro:domain} for a result on domains that makes additionally use of \emph{extension} theorems on domains.

With $f(x,v,A) =f(x,A)$, functional \eqref{intro:functional:grad} is a special case of \eqref{intro:functional} whenever $\A=\curl$, as $\curl u =0$ if and only if $u = \nabla v$ for zero-average functions $u $ on the $n$-torus. Therefore, it is natural to ask the same questions as for the gradient case, i.e. higher differentiability and higher integrability.
The question of (partial) regularity, i.e. $C^{0,\alpha}_{\mathrm{loc}}$ regularity of minimisers on $\Omega \setminus \Omega^s$ up to a singular set $\Omega^s$ of zero Lebesgue measure has, for instance, been studied in \cite{CG,Gmeineder,RL}. In contrast, in this work we are focused on obtaining higher integrability of minimisers to the functional $I$.
\medskip

In contrast to other contributions, our assumptions on the integral $f$ will be rather modest. To handle the functional \eqref{intro:functional}, we only assume that $f$ satisfies the growth condition
\begin{equation} \label{intro:growth}
 \vert w \vert^p \leq f(x,w) \leq \nu \vert w \vert^p +c.
 \end{equation}
In particular, we highlight that in contrast to many other contribution we do \emph{not need}
\begin{enumerate} [label=(\roman*)]
    \item the assumption that $f$ is already $\A$-quasiconvex or convex. In the simplest setting (cf. Theorem \ref{thm:1}) we even do \emph{not} need any continuity assumptions on $f(x,\cdot)$;
    \item any continuity assumptions on $f(\cdot,w)$.
\end{enumerate}
Our strategy of proof of course bears similarities to other results in this direction, but, according to our lax assumptions, some new ideas are needed. In particular, the proof is not based on a version of Caccioppoli's inequality (cf. \cite{Evans}). First, such an inequality is already harder to obtain in an $\A$-free setting (also cf. \cite{RL}) but, second and more importantly, due to our non-restrictive assumptions on $f$ it is hard to imagine that a Caccioppoli-based approach works.

Instead, the key idea that we are going to employ is the method of $\A$-free truncation: We cut-off the concentrations of a minimiser $v$ to obtain an $\A$-free $L^{\infty}$ function $\tilde{v}$. As $v$ minimises the functional $I$, we know that $I(v) \leq I(\tilde{v})$. On the other hand, the fact that $\tilde{v} \in L^{\infty}$ gives good upper bounds on $I(\tilde{v})-I(v)$. Combining the bounds of both directions gives an integrability estimate.

We remark that a similar methods have been used for instance in \cite{KL,Lewis} to obtain higher integrability to solutions of elliptic and parabolic differential equations.
\medskip

 We start by giving a definition of the truncation property \eqref{trunc} that plays a crucial role throughout the paper; note that $\M$ denotes the Hardy-Littlewood maximal function.

\begin{definition}
   Let $1 \leq p < \infty$. We say that $\A$ satisfies the property \eqref{trunc} if the following holds: There is a constant $C(\A)$, such that for any $u \in L^p(\T_n;\R^m)$ with $\A u =0$ and any $\lambda >0$ there exists an $\tilde{u} \in L^{\infty}(\T_n;\R^m)$ such
   that $\A \tilde{u}=0$ and
   \begin{equation}
   \label{trunc} \tag{TP}
   \begin{cases}
       \Vert \tilde{u} \Vert_{L^{\infty}} \leq C(\A) \lambda, \\
       \{ u \neq \tilde{u}\} \subset \{ \M u \geq \lambda \}.
   \end{cases}
   \end{equation}
   \end{definition}

While the family of operators, where such a truncation operator exists, is not fully classified, we can still prove \eqref{trunc} for many physical operators  So far, such a property has been shown for $\A= \curl$ (again through the identification $\curl u= 0 \Leftrightarrow u = \nabla v$) \cite{AF,Liu,Zhang}, $\A = \mathrm{d}$ (the exterior derivative of a differential form) \cite{Schiffer}, $\A = \divergence$ when acting on \emph{symmetric} $3 \times 3$ matrices \cite{BGS} and implicitly for some other operators as well \cite{BDF,BDS,Schiffer2}.
The first main result then reads as follows:
\begin{thm} \label{thm:1}
    Let $1<p<\infty$. Suppose that $f$ obeys the growth property \eqref{intro:growth} and that $\A$ satisfies \eqref{trunc}. Let $u \in L^p(\T_n;\R^m)$ be a minimiser to $I$. Then there is an $\varepsilon_0 = \varepsilon_0(p,C(\A),\nu)$ such that for all $\varepsilon < \varepsilon_0$ we have $u \in L^{p+\varepsilon}(\T_n;\R^m)$.
\end{thm}

As mentioned, the results presented in this note are quite conditional in the sense that we \emph{must} have the truncation property to achieve the result. The proofs of those truncation statements are, however, quite involved and we skip an in-detail description on \emph{how} property \eqref{trunc} is achieved.

The reason why the result of Theorem \ref{thm:1} is stated on the torus, is that the truncation statement \eqref{trunc} is usually shown either on the torus or on the full space and \emph{not} on domains, as we have to additionally deal with boundary behaviour there. Nevertheless, we also show (cf. Corollary \ref{coro:domain}), that the statement of Theorem \ref{thm:1} may be generalised to certain domains. Again, this statement is fairly conditional; here we assume a further \emph{extension} property for the differential operator $\A$. In particular, we can show that a truncation property like \eqref{trunc} may hold on a domain $\Omega$ instead of $\T_n$ as follows: Supposing $\Omega \subset \T_n$ we extend the function $u$ to to $\T_n$, truncate on the torus and restrict back to $\Omega$.

Again, proofs of those extension property are rather elaborate and we refer to \cite{GS} and \cite{BS,GmR} for more details.
\medskip

Another growth condition of interest (that for example appears in the study of quasiconformal maps (cf. \cite{Astala,Faraco,Gehring,IS,LN})) is the following. Let $\Pi$ be a homogeneous $\A$-quasiaffine map of $p$-growth that obeys
\[
\vert \Pi(w) \vert \leq \vert v \vert^p
\]
and let $\alpha>1$ such that
\begin{equation} \label{intro:growth:2}
    \vert w \vert^p - \alpha \Pi(w) \leq f(x,w) \leq \nu \vert w \vert^p +c.
\end{equation}
Examples of such $\Pi$ are for instance minors if $\A= \curl$ acting on $u \colon \T_n \to \R^{n \times l}$, in particular $\Pi=\mathrm{det}$ for $n \times n$ matrices.

Suppose further that $f$ satisfies the continuity condition
\begin{equation} \label{intro:conti}
    \vert f(x,w) - f(x,w') \vert \leq L( 1 + \vert w \vert^{p-1} + \vert w' \vert^{p-1}) \vert w -w' \vert.
\end{equation}
We remark that this condition is natural to assume if $f$ is $\A$-quasiconvex \cite{FM,GR,MP}.
Consider the functional
\begin{equation} \label{intro:functional:2}
    I'(u) = \begin{cases}
        \int_{\T_n} f(x,u(x)) \dx & \text{if } \A u =0, \text{ and } \int_{\T_n} u=u_0, \\
        \infty & \text{else}
    \end{cases}
\end{equation}
where $u_0 \in \R$ is some predefined average.
\begin{thm} \label{thm:3}
    Let $1 <p<\infty$. Suppose that $f$ obeys the growth property \eqref{intro:growth:2} and \eqref{intro:conti} and that $\A$ satisfies \eqref{trunc}. Let $u \in L^p(\T_n;\R^d)$ be a minimiser to $I'$. Then there is an $\varepsilon_0 = \varepsilon_0(L,p,C(\A),\nu+\alpha)$ such that for all $0<\varepsilon< \varepsilon_0$ we have $u \in L^{p+\varepsilon}(\T_n;\R^m)$.
\end{thm}
We mention that the condition $\int_{\T_n} u = u_0$ is a natural replacement of boundary conditions on the torus $\T_n$ and, moreover, $\A$-quasiconvexity together with coercivity condition \eqref{intro:growth:2} again guarantees existence of minimisers.

Before continuing with the main part of the article, let us mention that the results obtained here are fully qualitative: We loose constants at many steps of the estimates and therefore, the value of $\varepsilon_0$ obtained in Theorems \ref{thm:1} and \ref{thm:3} is certainly non-optimal.

\subsection*{Structure}
In Section \ref{sec:2} we introduce notation and gather some auxiliary results. Section \ref{sec:3} is concerned with showing how higher integrability is obtained by a 'reverse' $L^p$-$L^1$ estimate of the function on superlevel sets. In the remainder of the article, i.e. Section \ref{sec:4}, we therefore only need to show this reverse estimate; this is achieved with aforementioned truncation property \eqref{trunc}.
\subsection*{Acknowledgements}
The author would like to thank Franz Gmeineder for fruitful discussions and comments.
\section{Preliminaries} \label{sec:2}
In this section we gather set up some notation and gather some preliminary results. We start with some properties of the differential operator $\A$. We assume that $\A$ is a linear, constant coefficient differential operator
\[
\A u = \sum_{\vert \alpha \vert =k} A_{\alpha} \partial^{\alpha} u
\]
where $A_{\alpha} \in \Lin(\R^m;\R^l)$ are linear maps. The adjoint of $\A$ is defined via
 \[
 \A^{\ast} \phi = \sum_{\vert \alpha \vert =k} A_{\alpha}^T \partial^{\alpha} \phi
 \]
 and we call an $L^p$-function $\A$-free if for any $\phi \in C_c^{\infty}(\Omega;\R^m)$ (or $\phi \in C^{\infty}(\T_n;\R^m)$ we have $\langle u , \A^{\ast} \phi \rangle=0$.
We define
\[
\A[\xi] = \sum_{\vert \alpha \vert =k} A_{\alpha} \partial^{\alpha} u
\]
to be the Fourier symbol of $\A$. Following \cite{FM,GR} we now can define the notion of $\A$-quasiconvexity.
\begin{definition}
    We call a measurable and locally bounded function $f \colon \R^m \to \R$ $\A$-quasiconvex if for all $w \in \R^m$ and all
    $$\psi \in \test = \{\phi \in C^{\infty}(\T_n;\R^m) \cap \ker \A \colon \int_{\T_n} \psi=0\}$$
    we have 
    \[
    f(w) \leq \int_{\T_n} f(w +\psi(x)) \dx.
    \]
    We call $f$ $\A$-quasiaffine, if both $f$ and $-f$ are $\A$-quasiconvex.
    \end{definition}
While we do not need the notion of $\A$-quasiconvexity in this note, let us mention that under the growth condition \eqref{intro:growth}, $\A$-quasiconvexity of $f(x,\cdot)$ for a.e. $x$ is a necessary and sufficient condition for lower semicontinuity and hence, by the direct method, a sufficient condition for existence of minimisers. Let us further mention that $\A$-quasiconvexity together with the upper growth condition of \eqref{intro:growth} yields (cf. \cite{GR,MP})
\[
\vert f(x,w) - f(x',w) \vert \leq C(1 + \vert w \vert^{p-1} + \vert w' \vert^{p-1}) \vert w - w' \vert
\]
for $w -w' \in \spann_{\xi \in \R^n \setminus \{0\}} \ker \A[\xi]$ (this span is, in most examples, the full space $\R^m$).

The condition of $\A$-quasiaffinity is much sharper (we refer to \cite{BCO,GR} for more details). In particular, we only need that an $\A$-quasiaffine function $\Pi$ satisfies
\[
\Pi( w) = \int_{\T_n} \Pi(w+\psi(x)) \dx
\]
for all $\psi \in \test$ and that $\Pi$ is a polynomial (as long as $\spann_{\xi \in \R^n \setminus \{0\}} \A[\xi] = \R^m$). In particular, any homogeneous component of $\Pi$ is itself $\A$-quasiaffine, and thus we assume that
\[
\Pi \text{ is an $\A$-quasiaffine polynomial of degree $p \in \N$}
\]
and (by scaling) we might assume that $\vert \Pi(w) \vert \leq \vert w \vert^p$.

Finally, we also recall some elementary, yet important facts about the Hardy-Littlewood maximal function. It is defined pointwisely by
\[
\M u (x) = \sup_{r>0} \tfrac{1}{L^n(B_r(x))} \int_{B_r(x)} \vert u \vert \dx.
\]
Here, for the purpose of this definition, a function $u \in L^1(\T_n;\R^m)$ shall be understood as a $\Z^n$-periodic function in $L^1_{\mathrm{loc}}(\R^n;\R^m)$.

The maximal function is a sublinear operator that is bounded from $L^p$ to $L^p$ for any $p \in (1,\infty]$ and bounded from $L^1$ to $L^{1,\infty}$. Moreover, obviously $\vert u \vert \leq \M u $ almost everywhere. The property we use the most in this paper is, however, the following lemma (cf. \cite{Zhang}) that is a direct consequence of the boundedness of the maximal operator from $L^1$ to $L^{1,\infty}$.
\begin{lemma} \label{LZL}
Let $u \in L^1(\T_n;\R^m)$. There exists a dimensional constant $c(n)>0$, such that for any $\lambda>0$ we have
\begin{equation} \label{ZL}
\mathcal{L}^n(\{\M u \geq \lambda\}) \leq c(n) \lambda^{-1} \int_{\{\vert u \vert \geq \lambda/2\}} \vert u \vert \dx.
\end{equation}
\end{lemma}

\section{A hole filling lemma} \label{sec:3}
We start by showing how an estimate that arises out of the truncation technique (cf. the next Section \ref{sec:4}) yields higher integrability. The proof is quite reminiscent of Widman's hole filling technique \cite[Section 1.2.3]{BF}, \cite{Widman}, which has been widely used in regularity theory of PDEs.

\begin{lemma} \label{lemma:Widman}
Suppose that $u \in L^p(\T_n;\R^m)$ and that there exists $\lambda_0>0$, $C,R>0$ such that for all $\lambda>\lambda_0$
\begin{equation} \label{eq:Widman}
    \int_{\{ \vert u \vert \geq R \lambda\}} \vert u \vert^p \dx \leq C \lambda^{p-1} \int_{\{\vert u \vert  \geq \lambda\}} \vert u \vert \dx.
\end{equation}
Then there is $\varepsilon_0 = \varepsilon_0(R,C,p)$ such that $u \in L^{p+\varepsilon}$ for all $\varepsilon<\varepsilon_0$.
\end{lemma}
The proof is quite short, we shortly stress the key idea here. Suppose for simplicity that $u = \chi_E \cdot S \lambda$ for some measurable set $E$ and $S>1$. Then the inequality \eqref{eq:Widman} scales like $S^p$ and like $S^1$ on the left- and right-hand-side, respectively. Therefore, $S$ cannot be too large. Translating this heuristic argument into an inequality, we get 

 \begin{align*}
     \int_{\{\vert u \vert \geq S \lambda\}} \vert u \vert^p \dx \leq \mu \int_{\vert u \vert \geq \lambda \}} \vert u \vert^p \dx.
 \end{align*}
for some $\mu<1$, i.e. we get additional decay. This additional decay is then used to obtain higher integrability.

\begin{proof}
 Define $S= \max\{ (2C)^{1/(p-1)},R\}$. Then we may estimate for any $\lambda> \lambda_0$
 \begin{align*}
     C \lambda^{p-1} \int_{\{ \vert u \vert \geq \lambda \}} \vert u \vert 
     &
     = C \lambda^{p-1} \int_{\{ S \lambda \geq \vert u \vert \geq \lambda \}} \vert u \vert \dx +  C \lambda^{p-1}  \int_{\{\vert u \vert \geq S \lambda  \}} \vert u \vert \dx \\
     &\leq C \int_{\{ S \lambda \geq \vert u \vert \geq \lambda \}} \vert u \vert^p \dx + \tfrac{1}{2} \int_{\{\vert u \vert \geq S \lambda\}} \vert u \vert^ p \dx
 \end{align*}
 and therefore
 \begin{align*}
     \int_{\{\vert u \vert \geq S \lambda\}} \vert u \vert^p \dx \leq 2C \int_{\{S\lambda \geq \vert u \vert \geq \lambda \}} \vert u \vert^p \dx.
 \end{align*}
 Adding $2C$ times the left-hand side to the equation implies
 \begin{equation} \label{eq:Widman:2}
 \int_{\{\vert u \vert \geq S \lambda\}} \vert u \vert^p \dx \leq \tfrac{2C}{2C+1} \int_{\{ \vert u \vert \geq \lambda\}} \vert u \vert^p dx.
 \end{equation}
 Applying this inequality successively yields for each $r \in \N$
 \[
 \int_{\{S^{r+1} \lambda_0 \geq \vert u \vert \geq S^r \lambda_0\}} \vert u \vert^p \dx \leq \left(\tfrac{2C}{2C+1}\right)^r \Vert u \Vert_{L^p}^p
 \]
 and, therefore, for $\varepsilon< \tfrac{\ln((2C+1)/(2C)}{p\ln(S)}$ we have
 \[
 \int_{\{ S^{r+1} \lambda_0 \geq \vert u \vert \geq S^r \lambda_0\}} \vert u \vert^{p+\varepsilon} \dx \leq S^{(r+1)p \varepsilon} \int_{\{S^{r+1} \lambda_0 \geq \vert u \vert \geq S^r \lambda_0\}} \vert u \vert^p \dx  \leq S^{p \varepsilon} \theta^r \Vert u \Vert_{L^p}^p
 \]
for $\theta = \tfrac{2C}{C+1} S^{\varepsilon}<1$. Therefore, summing over all $r \in \N$ and using that $u$ is an $L^p$ function, we obtain
\[
\int_{\{\vert u \vert \geq \lambda_0\}} \vert u \vert^{p + \varepsilon} \dx < \infty
\]
and, thus, $u \in L^{p+\varepsilon}(\T_n;\R^m)$.
\end{proof}
\section{Proofs via truncation} \label{sec:4}
The goal of this section is to obtain the proofs of Theorems \ref{thm:1} and \ref{thm:3}. In view of Lemma \ref{lemma:Widman}, it suffices to arrive estimate at \eqref{eq:Widman}.
\subsection{Proof of Theorem \ref{thm:1}}
As said, the remaining part in order to prove Theorem \ref{thm:1} is the following statement.
\begin{lemma} \label{lemma:kasparov}
    Suppose that $f$ obeys \eqref{intro:growth} and $\A$ satisfies the truncation property \eqref{trunc} and that $u$ is a minimiser to $I$ in $L^p(\T_n;\R^m)$. Then the estimate \eqref{eq:Widman} holds for $R=2$, $C=C(\nu,C(\A))$ and $\lambda>\lambda_0(\nu,C(\A),c)>0$.
\end{lemma}
\begin{proof}
Let $u$ be a minimiser and let $\tilde{u}$ be the truncation according to \eqref{trunc} at level $\lambda>0$. On the one hand,
\begin{equation} \label{kasparov:1}
I(\tilde{u}) -I(u) \geq 0   
\end{equation}
as $u$ is a minimiser and, on the other hand, we have
\begin{equation} \label{kasparov:2}
\begin{split}
I(\tilde{u}) - I(u) &= \int_{\{ \M u \geq \lambda \}} f(x,\tilde{u}) - f(x,u) \dx  \\ & \leq - \int_{\{\M u \geq \lambda\}} \vert u \vert^p \dx + \mathcal{L}^n(\{\M u \geq \lambda \}) \left(C(\A)^p \lambda^p \nu + c \right).
\end{split} 
\end{equation}
Combining \eqref{kasparov:1} and \eqref{kasparov:2} yields
\begin{equation} \label{kasparov:3}
     \int_{\{\M u \geq \lambda\}} \vert u \vert^p \dx \leq (\lambda^p \nu C(\A)^p +c) \mathcal{L}^n(\{ \M u \geq \lambda\})
\end{equation}
Then using that $\vert u \vert \leq \M u$, for $\lambda > \lambda_0:= \frac{c^{1/p}}{C(\A)\nu^{1/p}}$ we have
\[
 \int_{\{ \vert u \vert \geq \lambda\}} \vert u \vert^p \leq 2 \lambda^p \nu C(\A)^p \mathcal{L}^n(\{ \M u \geq \lambda\}).
\]
Finally, using estimate \eqref{ZL} from Lemma \ref{LZL} yields
\begin{equation} \label{kasparov:finale}
\int_{\{\vert u \vert \geq \lambda \}} \vert u \vert^p \dx \geq \lambda^{p-1} 2 \nu C(\A)^p C(n) \int_{\{\vert u \vert \geq \lambda/2\}} \vert u \vert \dx.
\end{equation}
\end{proof} 
Together with Lemma \ref{lemma:Widman} this concludes the proof of Theorem \ref{thm:1}. Before continuing with Theorem \ref{thm:3}, we shortly pause and point out some observations.
\begin{remark}
    \begin{enumerate}[label=(\roman*)]
        \item Due to the choice of $S$ at the start of Lemma \ref{lemma:Widman}, the gain $\varepsilon$ in integrability tends to zero, when $p$ approaches $1$. As for $p=1$ improvement in integrability cannot be shown, the method of proof is qualitatively optimal (also cf. \cite{GK1}).
        \item The method of truncation can be also tackled to consider integrands that have the form
        \[
            \int_{\T_n} f(x,v(x),\mathbb{A}v(x)) \dx,
        \]
        whenever $\mathbb{A}$ is a \emph{potential} of $\A$, i.e. (for zero average functions) $\A u =0 \Leftrightarrow u = \mathbb{A} v$. In particular, consider $\mathbb{A}= \nabla$. Using quantitative versions of  Lipschitz truncation, then one also might obtain better estimates for the gain of integrability.
        \item The method of proof also works with minor modifications if we add an additional linear error
        \[
        I(u) = \begin{cases}
        \int_{\T_n} f(x,u(x)) - g \cdot u \dx & \quad \text{if } \A u =0 \\
        \infty & \text{else.}
        \end{cases}
        \]
        for $g \in L^{r}(\Omega;\R^m)$ for $r>p'$; the only difference is the slight adjustment of \eqref{kasparov:2}.
        \item A recent point of study has been functionals, of $p$-$q$ growth, that is 
        \[
        \vert w \vert^p \leq f(x,w) \leq \nu \vert w \vert^q +c
        \]
        for $1<p<q<\infty$. Without additional assumptions (see mentioned references), due to the different homogeneity in $\lambda$, the truncation method is not able to show any improved regularity. We stress that many of the results there have been shown using stronger assumptions on the integrands, for instance that $f(x,\cdot)$ is \emph{convex}. Even in the quasiconvex setting, higher integrability is not so clear. Without hope of giving a truly complete list of references, we still refer to\cite{BS,CKP,ELM,GK2,Koch,Marcellini1,Marcellini2,Schmidt} for some results on the case of $p$-$q$ growth.
        \item As a replacement of the truncation property \eqref{trunc}, another idea (on the torus) would be to construct $\tilde{u}$ as follows: Simply cut of $u$ to obtain
        \[
        \bar{u} = \begin{cases}
            u(x) & \text{if }\vert u \vert \leq \lambda, \\
            0 & \text{else,}
        \end{cases}
        \]
        and then apply a projection operator (cf. \cite{FM,GR}) onto the kernel of $\A$ to obtain some $\tilde{u}$. The problem is, however, that this projection is non-local \emph{and} only maps $L^r$ to $L^r$ for $1<r<\infty$; In particular, $\tilde{u}$ \emph{is not} in $L^{\infty}$. Even if imposing an additional (natural) continuity assumption
        \[
        \vert f(x,w)-f(x,w') \vert \leq C(1+\vert w \vert^{p-1} + \vert w' \vert)^{p-1} \vert w- w' \vert,
        \]
        the simple truncation method does not yield any higher integrability. In particular, using the same estimates as before, we obtain only 
        \[
            \int_{\{\vert u \vert \geq \lambda\}} \vert u \vert^p \dx \leq C(s) \lambda^{p-1-\frac{p}{s'}} \left(\int_{\{ \vert u \vert \geq \lambda \}} \vert u \vert^s \right)^{1/s},
        \]
       for any $s$ that is \emph{strictly} larger than one; but this estimate alone does not yield any higher integrability.
    \end{enumerate}
        \end{remark}

We finish this subsection by further commenting on the case, where we are on a domain $\Omega \subset \R^n$. It is clear that the proof never used the fact that the domain is actually $\T_n$, it is only the truncation property \eqref{trunc} that uses the special structure of the torus. Therefore, one would need a truncation that also works on $\Omega \subset \R^n$.

If the minimisation problem further prescribes (for instance zero) boundary values, then a truncation that preserves boundary values (cf. \cite{DSSV,FJM} for more details) is needed. If, however, we only consider
\begin{equation} \label{eq:func:omega}
I(u) = \begin{cases}\int_{\Omega} f(x,u(x)) \dx& \text{if } \A u =0, \\
\infty & \text{else,}
\end{cases}
\end{equation}
then the strategy might be as follows: We extend $u$ to the full space (or to a rescaled torus) while preserving the constraint $\A u=0$. We then may use the truncation theorem to get a truncated version on $\Omega$.

By rescaling, we now suppose that $\Omega \subset \T_n$ is a subset of the torus. Again, we may define the maximal function for $ u \in L^1(\Omega)$ by extending $u$ by $0$ to $\T_n$.
\begin{coro} \label{coro:domain}
    Let $1 < p< \infty$. Suppose that $f$ obeys growth property \eqref{intro:growth} for a.e. $x \in \Omega$ and that $\A$ satisfies \eqref{trunc}. Suppose further that $\Omega$ satisfies the following extension property: \\
    There are $C_1,C_2>0$ (depending on $\Omega$) such that for any $v \in L^p(\Omega;\R^m)$ with $\A v=0$ there is a $\mathscr{E}v \in L^p(\Omega;\R^m)$ with $\A v=0$ such that either
    \begin{equation} \label{ext}
        \mathcal{L}^n(\{\M (\mathscr{E}v) \geq \lambda\}) \leq C_1(\Omega) \mathcal{L}^n(\{\M (\mathscr{E} v) \leq C_2(\Omega) \lambda\})
    \end{equation}
    or \begin{equation} \label{ext2}
        \M \E v (x) \leq C_3(\Omega) \M v(x) \quad \text{for a.e. } x \in \Omega.
    \end{equation}
    Suppose that $u \in L^p(\Omega;\R^m)$ is a minimiser to $I$ from \eqref{eq:func:omega}. Then there is a $\varepsilon_0= \varepsilon(p,C(\A),\nu,\Omega)>0$ such that for all $0<\varepsilon< \varepsilon_0$ we have $u \in L^{p+\varepsilon}(\Omega;\R^m)$.
\end{coro}
\begin{proof}
    Let $u$ be a minimiser to $I$. We truncate the extended version $u$ and obtain as in the proof of Lemma \ref{lemma:kasparov} (for $\lambda>\lambda_0$)
    \[
        \int_{\{M (\mathscr{E} u) \geq \lambda \} \cap \Omega} \vert u \vert^p \leq C\lambda^p \mathcal{L}^n ( \{ \M (\E u) \geq \lambda\}).
    \]
   Emplyoing our assumption on the lemma we then also get \eqref{kasparov:finale}, i.e. if we assume that \eqref{ext} holds then
    \[
    \int_{\{\vert u \vert \geq \lambda\}} \vert u \vert^p \geq \lambda^{p-1}C C_1(\Omega)C(n) \int_{\{\vert u \vert \geq C_2(\Omega)\lambda/2\}} \dx.
    \]
    and otherwise, if \eqref{ext2} holds we have
    \[
          \int_{\{\vert u \vert \geq \lambda\}} \vert u \vert^p \geq \lambda^{p-1}C C_1(\Omega)C(n) \int_{\{\vert u \vert \geq \lambda/(2C_3(\Omega))\}} \dx.
    \]
    Due to Lemma \ref{lemma:Widman}, this is sufficient to obtain higher integrability.
\end{proof}
We finish the discussion of the strong coercivity condition with some remarks on the extension property.

First, we consider an example of an extension  that obeys the conditions of above Corollary. Let $\A = \divergence$ and $\Omega=(0,1/2)^n$ be the cube. By reflection, if $x_n > 1/2$ define
\[
\tilde{u}(x) = \left(\begin{array}{c}
    -u_1(Sx) \\
   \ldots \\ 
   -u_{n-1}(Sx) \\
    u_n(Sx)
\end{array} \right), \quad Sx= (x_1,x_2,\ldots,x_{n-1}, 1- x_n)
\]
one might extend $u$ to $(0,1/2)^{n-1} \times (0,1)$. By repeating the procedure one can extend $u$ to $(0,1)^{n}$ to some $\E u$. Identifying $(0,1)^n$ with $\T_n$, we can check that this $\E u$ is still divergence-free on $\T_n$.
On the other hand, by considering the definition of the maximal function we can see that 
\[
\M (\E u ) \leq 2^n \M u
\]
and by the second extension assumption of Corollary \ref{coro:domain} we also get higher integrability.

Similar to the topic of $\A$-free truncation, however, $\A$-free extension is not fully resolved and we refer to the articles \cite{GS,Kato} for more information. We stress however that we have the following result based on an $\A$-free extension result of \cite{GS}.
\begin{lemma}
Suppose that there is a bounded linear operator $\E \colon L^1(\Omega;\R^m) \to L^1(\T_n;\R^m)$ with the following properties:
\begin{enumerate}[label=(\roman*)]
    \item \textbf{conservation of the differential constraint:} $\A u =0$ in $\D'(\Omega;\R^l)$ implies $\A (\E u) =0$ in $\D'(\T_n;\R^l)$;
    \item \textbf{extension:} $\E u = u$ pointwise a.e. in $\Omega$;
    \item \textbf{local bounds:} There exists $c_1,c_2>0$ such that for any for all $y \in \Omega^C$, we have
    \[
    \int_{B_{\rho/2}(y)} \vert \E u \vert \dx \leq \int_{B_{c_1\rho} \cap S_{\rho}(\Omega)} \vert u \vert \dx,
    \]
    where $\rho =\dist(y,\partial \Omega)$ and 
    \[
    S_{\rho}= \{ x \in \Omega \colon c_2^{-1} \rho \leq \dist(x,\partial \Omega) \leq c_2 \rho \}.
    \]
\end{enumerate}
Then the extension satisfies property \eqref{ext2}.
\end{lemma}
These local bounds can be achieved in the most classic case of $\A=\curl$, i.e. with the gradient as a potential, for Lipschitz domains, cf. \cite{Jones,Stein}. Moreover, again for Lipschitz domains, such local bounds may be obtained for annihilators of $\C$-elliptic operators \cite{BC,GmR} or for $\A= \divergence$ \cite{GS}, making Corollary \ref{coro:domain} valid for those cases.
\begin{proof}
    Let $x^{\ast} \in \Omega$ and $r>0$. If $r < \dist(x^{\ast},\Omega)$ then 
    \[
    \int_{B_r(x^{\ast}) } \vert \E u \vert \dx =\int_{B_r(x^{\ast}) \cap \Omega} \vert u \vert \dx.
    \]
    Suppose now $r> \dist(x^{\ast},\Omega)$ and consider $A= B_r(x^{\ast}) \setminus \bar{\Omega}$. By Besicovitch covering theorem we can cover $A$ by balls
    \[
    B_{\dist(y,\partial \Omega)/2} (y)
    \]
    for $y \in A$ so that any ball intersects with at most $c(n)$ other balls. Let $\mathcal{B}$ be the collection of those balls $B$ that have radii $r(B)$.
    Therefore,
    \[
    \int_A \vert \E u \vert \dx \leq \sum_{ B \in \B} \int_B \vert \E u \vert \dx .
    \]
    Now applying the local bounds we obtain
    \[
    \sum_{B \in \B} \int_B \vert \E u \vert \dx = \sum_{B \in \B} \int_{c_1 B \cap S_{r(B)}} \vert u \vert \dx = \int_{\Omega} \vert u \vert \# \{B \colon x \in c_1 B \cap S_{r(B)}\} \dx.
    \]
    We now estimate the number of those balls for any $x \in \Omega$. Observe that if $x \notin B_{c_1r}(x^{\ast})$  we have $\# \{B \colon x \in c_1 B \cap S_{r(B)}\}=0$. Else, denote by $\delta= \dist(x,\partial \Omega)$. If $x \in c_1 B \cap S_{r(B)}$ we know that $c_2^{-1}\delta \leq r(B) \leq c_2{\delta}$ and that $B \subset B_{c_1c_2\delta}(x)$. As only $c(n)$ balls intersect at every point, we therefore get
    \[
        c(n) \mathcal{L}^n(B_{c_1c_2\delta})(x) \geq \sum_{B \colon x \in c_1 B \cap S_{r(B)}} \mathcal{L}^n(B)
    \]
    and, using the lower bound on the radius of the balls we conclude
    \[
    \#\{B \colon x \in c_1B \cap S_{r(B)} \} \leq c(n) (c_1c_2^2)^{n} =:C.
    \]
    and thus
    \[
    \int_{B_r(x^{\ast})} \vert \E u \vert \leq C\int_{B_{c_1r}(x^{ast}) \cap \Omega} \vert u \vert \dx.
    \]
    Taking the supremum over all $r>0$ we finally arrive at the desired estimate
    \[
    \M (\E u) (x) \leq c_1^n C \M(u)(x) \quad \text{for all } x \in \Omega.
    \]
    \end{proof} 
    For previously mentioned operators and Lipschitz domains $\Omega$, we actually have this estimate (e.g. \cite[Lemma 5.9. ff.]{GS}such that, in this case, the higher integrability result on the torus may be extended to any Lipschitz domain (when \emph{not} imposing additional boundary conditions).

\subsection{Proof of Theorem \ref{thm:3}}
Again, we only need to achieve the estimate \eqref{eq:Widman}, which is reflected in the following lemma.

\begin{lemma} \label{lemma:anand}
Suppose that $f$ obeys \eqref{intro:growth:2} and \eqref{intro:conti} and $\A$ satisfies the truncation property \eqref{trunc}. Let $u_0 \in \R$ and $u$ be a minimiser to $I$ in $L^p(\T_n;\R^m)$. Then the estimate \eqref{eq:Widman} holds for $R=2$ and $C=C(\nu+\alpha,C(\A),L)$ for $\lambda > \lambda_0(u_0,\nu+\alpha,C(\A),c,L)$.
\end{lemma}
\begin{proof}
    Let $u$ be a minimiser and let $\tilde{u}$ be the truncation according to \eqref{trunc} at level $\lambda>0$. Define 
    \[
    \bar{u} = \tilde{u} + (u_0 - \int_{\T_n} \tilde{u} \dx).
    \]
    Then $\bar{u}$ is an admissible competitor to $u$ and thus
    \begin{equation} \label{anand:1}
        0 \leq I(\bar{u}) -I(u) \leq (I(\bar{u}) - I(\tilde{u})) + ( I(\tilde{u}) - I(u)).
    \end{equation}
    We estimate both summands from above: We have 
    \[
    \left \vert u_0 - \int_{\T_n }\tilde{u} \dx \right \vert \leq \int_{\T_n} \vert u - \tilde{u} \vert \dx  \leq C(\A) \lambda \mathcal{L}^n(\{u \geq \lambda\}) + \int_{\{\vert u \vert \geq \lambda\}} \vert u \vert \dx
    \]
    and, therefore, due to continuity property \eqref{intro:conti}
    \begin{equation} \label{anand:2}
        I(\bar{u}) - I(\tilde{u}) \leq C(L,\A) (1+ \lambda^{p-1}) \int_{\{\vert u \vert \geq \lambda\}} \vert u \vert \dx
    \end{equation}
    For $I(\tilde{u})-I(u)$ observe that
    \begin{align*}
        I(\tilde{u})-I(u) &= \int_{\{ \M u \geq \lambda\}} f(x,\tilde{u}(x)) - f(x,u(x)) \dx \\
        & \leq (\nu \lambda^p+c) \mathcal{L}^n(\{\M u \geq \lambda\}) - \int_{\{  \M u  \geq \lambda \}} \vert u \vert^p  \dx + \int_{\{ \M u  \geq \lambda\} } \alpha \Pi(u) \dx
     \end{align*}
     The first two terms can be estimated as in the proof of Lemma \ref{lemma:kasparov}. For the last term note that
     \begin{equation*}
         \Pi(u_0) - \Pi \left(\int_{\T_n} \tilde{u} \dx \right) = \int_{\{ u \neq \tilde{u} \}} \Pi(u) - \Pi(\tilde{u}) \dx
     \end{equation*}
     and therefore
     \begin{equation} \label{anand:3}
     \begin{split}
         \int_{\{ \M u \geq \lambda\}} \alpha \Pi(u) &\leq \int_{\{ \M u \geq \lambda\}} \Pi(\tilde{u}) \dx  + \left( \Pi(u_0) - \Pi \left(\int_{\T_n} \tilde{u} \dx \right) \right)\\
         &\leq \mathcal{L}^n(\{ \M u \geq \lambda\}) C(\A)^p\lambda^p + (1+u_0^{p-1}) \int_{\{\vert u \vert \geq \lambda\}} \vert u \vert \dx 
    \end{split}
     \end{equation}
     Using \eqref{anand:1}-\eqref{anand:2} we arrive at
     \[
     \int_{\{\vert u \vert \geq \lambda \}} \vert u \vert^p \dx \leq  C(1+C(\nu+\alpha,C(\A))\lambda^p) \mathcal{L}^n(\{ \M u \geq \lambda\}.
     \]
    As in the proof of Lemma \ref{lemma:kasparov}, the estimate \eqref{ZL} for the maximal function and choosing $\lambda> \lambda_0$ large enough yields for $\lambda> \lambda_0$
    \begin{equation*}
          \int_{\{\vert u \vert \geq \lambda \}} \vert u \vert^p \dx  \leq C \lambda^{p-1} \mathcal{L}^n(\{ \M u \geq \lambda\}).
    \end{equation*}
\end{proof}
Last, we shortly comment about two cases where this type of growth condition is put into practice:
\begin{ex}
\begin{enumerate}[label=(\roman*)]
    \item A matrix $A \in \R^{n \times n}$ is $K$-quasiconformal, if there is a lower-bound on the determinant, i.e. $\det(A) \geq K \vert A \vert^n$. It has been shown that gradients of functions that are $K$-quasiconformal (meaning that their gradient is $K$-quasiconformal almost everywhere) obey higher integrability properties (with sharp upper bounds). In our case, we obtain a non-quantitative version of the higher integrability results (cf. \cite{Astala,Gehring}), as quasiconformal matrices are minimisers to the functional with integrand
    \[
    f(A) = \max\{0,\vert A \vert^n-K\det(A)\}.
    \]
    \item Higher integrability of minimisers of a $p$-Laplace (cf. \cite{CT,KMSX}) can be seen through two different viewpoints. First, a solution is a minimiser to
    \[
    I(u) = \int \vert \nabla u \vert^p \dx,
    \]
    and, again, we get a non-quantitative version through our results. On the other hand, write $\epsilon= \nabla u$. Then any solution to the $p$-Laplace equation may be seen as a minimiser to 
    \[
    \tilde{I}(\epsilon,\sigma) = \begin{cases}
        \int \tfrac{1}{p} \vert \epsilon \vert^p +\tfrac{1}{q} \vert \sigma \vert^q - \epsilon \cdot \sigma \dx & \text{if } \A (\epsilon,\sigma)=0, \\
        \infty & \text{else,} 
        \end{cases}
    \]
    where $\A(\epsilon,\sigma)$ is
    \[
    \A(\epsilon,\sigma) = \left( \begin{array}{c} \curl \epsilon \\
    \divergence \sigma
    \end{array}  \right).
    \]
    In particular, minimisers of $\tilde{I}$ are solutions to the differential inclusion
    \[
    \A(\epsilon,\sigma) =0 \quad \text{and} \quad \sigma= \vert \epsilon \vert^{p-2} \epsilon.
    \]
    With a slight modification adjusted to the different growths in $\epsilon$ and $\sigma$ one may also show higher integrability properties of solutions via the results of the second theorem.
\end{enumerate}
\end{ex}

\bibliography{biblio.bib}
\bibliographystyle{abbrv}

\end{document}